\documentclass[11pt]{article}

\usepackage{amssymb}
\usepackage{fnpct}
\usepackage{amsmath}
\usepackage[pdftex]{graphicx}
\usepackage{epsfig}
\usepackage{framed}
\usepackage{etoolbox}
\usepackage{fnpct}
\usepackage{overpic}
\usepackage{url}
\usepackage{color}
\usepackage[font=small,labelfont=bf]{caption} 
\usepackage{relsize} 
\usepackage{setspace}
\usepackage{hyperref}
\usepackage{tikz}
\usepackage{pdfpages}
\usetikzlibrary{arrows,positioning}
\usepackage{enumitem}

\usepackage{subfig}

\usepackage{tabularx}

\usepackage[utf8]{inputenc}

\newcommand{\R}{\mathbb{R}}

\def\hide #1 {}
\long\def\longhide #1 {}

\usepackage{amsthm}

\theoremstyle{plain}
\newtheorem{theorem}{Theorem}[section]
\newtheorem*{theorem*}{Theorem}
\newtheorem{lemma}[theorem]{Lemma}
\newtheorem{proposition}[theorem]{Proposition}

\theoremstyle{definition}
\newtheorem{remark}[theorem]{Remark}

\newtheorem{definition}[theorem]{Definition}

\theoremstyle{definition}

\newcommand{\E}{\mathbb{E}}

\newcounter{nootje}
\renewcommand{\check}[1]
  {\marginpar{\tiny\begin{minipage}{20mm}\begin{flushleft}\thenootje : #1\end{flushleft}\end{minipage}}\addtocounter{nootje}{1}}
\newcommand{\sjoerd}[1]{{#1}}
\newcommand{\sjoerdnew}[1]{{#1}}

\allowdisplaybreaks

\title{A resolution of the Gaussian hyperplane tessellation conjecture \\on the sphere}

\author{\begin{tabular}{ccc}
  Sjoerd Dirksen\thanks{Mathematical Institute, Utrecht University. {\sl Email}: s.dirksen@uu.nl.} & \quad \quad \quad \quad & Nigel Q. D. Strachan\thanks{Mathematical Institute, Utrecht University. {\sl Email}: n.q.d.strachan@uu.nl. }
  \end{tabular}
}
\date{}

\begin{document}
\maketitle
\begin{abstract}
\noindent 
We investigate how many hyperplanes with independent standard Gaussian directions one needs to produce a $\delta$-uniform tessellation of a subset $S$ of the Euclidean sphere, meaning that for any pair of points in $S$ the fraction of hyperplanes separating them corresponds to their geodesic distance up to an additive error $\delta$. \sjoerdnew{It} was conjectured that $\delta^{-2}w_*(S)^2$ Gaussian random hyperplanes are necessary and sufficient \sjoerdnew{for this purpose}, where $w_*(S)$ is the Gaussian complexity \sjoerdnew{of $S$}. We falsify this conjecture by constructing a set $S$ where $\delta^{-3}w_*(S)^2$ Gaussian hyperplanes are necessary and sufficient.
\end{abstract}

\section{Introduction}
In this paper we study uniform tessellations of subsets of the sphere. To any collection of vectors $g_1,\ldots,g_m\in \R^n$ we can associate the collection of homogenous hyperplanes 
$$g_i^{\perp} = \{x\in \R^n \ : \ \langle x, g_i\rangle=0\}, \qquad i=1,\ldots,m.$$
Each hyperplane naturally induces two halfspaces and hence, for any subset $S \subset \mathbb{S}^{n-1}$, the hyperplanes tessellate $S$. Letting $G\in \R^{m\times n}$ be the matrix with rows $g_i$, for any given $x\in S$ the vector
$$\mathrm{sign}(Gx):=(\mathrm{sign}(\langle x, g_i\rangle))_{i=1}^m \in \{-1,1\}^m$$
encodes in which cell of the tessellation $x$ is located. We say that the tessellation of $S$ is \emph{uniform} if, for any pair of points in $S$, the fraction of hyperplanes separating them corresponds to their distance up to a small error. 
\begin{definition}
Let $S \subset \mathbb{S}^{n-1}$ and $\delta > 0$. We say that $G\in \R^{m\times n}$ induces a \emph{$\delta$-uniform tessellation of $S$} if, for all $x,y \in S$,
$$\left|d_H(\mathrm{sign}(Gx),\mathrm{sign}(Gy)) - d_{\mathbb{S}^{n-1}}(x,y)\right| \leq \delta,$$
where $d_H(a,b) := \frac{1}{m} \sum_{i = 1}^m 1_{\{a_i \neq b_i\}}$ is the normalized Hamming distance on $\{-1,1\}^m$ and $d_{\mathbb{S}^{n-1}}(u,w) := \frac{1}{\pi} \arccos(\langle u, w\rangle)$ is the normalized geodesic distance on $\mathbb{S}^{n-1}$. 
\end{definition}
Uniform hyperplane tessellations were studied in depth by Plan and Vershynin in their seminal work \cite{plan2014dimension}. They derived a sufficient condition for the number of i.i.d.\ standard Gaussian hyperplanes that guarantees a uniform tessellation with high probability.  Their result is stated in terms of the \emph{Gaussian complexity} of a set $T\subset \R^n$ defined by 
$$w_*(T) := \mathbb{E} \left(\sup_{x \in T} |\langle x, g \rangle|\right),$$
where $g$ is standard Gaussian. 
\begin{theorem}[{\cite[Theorem~1.5]{plan2014dimension}}]
\label{thm:PVmain}
There exist constants $C,c>0$ such that the following holds. Let $S \subset \mathbb{S}^{n-1}$, $\delta > 0$, and let $G\in \R^{m\times n}$ be standard Gaussian. If 
$$m \geq C\delta^{-6} w_*(S)^2,$$
then $G$ induces a $\delta$-uniform tessellation of $S$ with probability at least $1 - 2 \exp(-c\delta^2m)$.
\end{theorem}
The authors of \cite{plan2014dimension} remarked that the dependence on $\delta$ in the previous result is sub-optimal and that they did not try to optimize the dependence. An improvement was obtained by Oymak and Recht in \cite{oymak2015near}. To formulate their result, let $N(T;\varepsilon)$ be the $\varepsilon$-covering number of a set $T \subset \mathbb{R}^n$ with respect to the Euclidean distance and let $B_2^n(0;\varepsilon)$ be the Euclidean ball in $\R^n$ with center $0$ and radius $\varepsilon$.
\begin{theorem}[{\cite[Theorem~2.5]{oymak2015near}}]
\label{thm:Oymak-Recht}
There exist constants $C,c_1,c_2>0$ such that the following holds. Let $S \subset \mathbb{S}^{n-1}$, $\delta > 0$, and let $G\in \R^{m\times n}$ be standard Gaussian. If $\varepsilon = c_1\frac{\delta}{\log(\delta^{-1})}$ and
$$m \geq C\delta^{-2} \log(N(S;\varepsilon)) + C\delta^{-3}w_*((S -S) \cap B_2^n(0;\varepsilon))^2,$$
then $G$ induces a $\delta$-uniform tessellation of $S$ with probability at least $1 - 2\exp(-c_2\delta^2 m)$. 
\end{theorem}
\sjoerdnew{Motivated by their result in Theorem~\ref{thm:PVmain}, Plan and Vershynin conjectured that $\delta^{-2}w_*(S)^2$ (not necessarily Gaussian) hyperplanes are necessary and sufficient to produce a $\delta$-uniform tessellation of $S$ \cite[Section 1.7]{plan2014dimension}. The follow-up work \cite{oymak2015near} more specifically conjectured that $\delta^{-2}w_*(S)^2$ Gaussian hyperplanes are necessary and sufficient and showed that this number of Gaussian hyperplanes is sufficient to induce a $\delta$-uniform tessellation for special structured sets (e.g., when $S$ is the intersection of $\mathbb{S}^{n-1}$ with a subspace).  
The purpose of \sjoerdnew{our} paper is to provide a negative answer to the latter conjecture} and to show that Theorem~\ref{thm:Oymak-Recht} cannot be improved further in general. \color{black} More precisely, we prove the following result. 
\begin{theorem}
\label{thm:main}
There exist constants $c>0$ and $c_1>c_2>0$ such that the following holds. Let $G\in \R^{m\times n}$ be standard Gaussian. For any $0<\delta<c$ there exists an $n=n(\delta) \in \mathbb{N}$ and $S=S(\delta)\subset \mathbb{S}^{n-1}$ such that if $m\geq c_1 \delta^{-3} w_*(S)^2$ then $G$ induces a $\delta$-uniform tessellation of $S$ with probability at least $0.99$ and if $m\leq c_2 \delta^{-3} w_*(S)^2$ then $G$ fails to induce a $\delta$-uniform tessellation of $S$ with probability at least $0.99$.
\end{theorem}

\subsection{Related work}

Uniform hyperplane tessellations of subsets of the sphere have proven useful for several purposes. First, they can be directly used for data dimension reduction: if $A$ induces a $\delta$-uniform tessellation on $S$, then the map $f(x)=\mathrm{sign}(Ax)$ is a $\delta$-binary embedding, meaning that it maps $S$ into the Hamming cube in a near-isometric manner (i.e., up to an additive error). In this setting, the number of hyperplanes $m$ governs the dimension reduction that can be achieved. For dimension reduction it is of interest to prove uniform hyperplane tessellation results for random matrices $A$ that support fast matrix-vector \color{black}multiplication, so that $f(x)$ can be computed in an efficient manner, see \cite{dirksen2018fast,oymak2017near,yi2015binary,yu2014circulant} and the references therein for results in this direction. Second, uniform hyperplane tessellations play a role in the theory of one-bit compressed sensing (see, e.g., \cite{boufounos20081,jacques2013robust,plan2013one}), the problem of reconstructing a signal $x$ from its one-bit measurements $\mathrm{sign}(Ax)$, which arise by quantizing analog measurements (represented by $Ax$) using a simple one-bit analog-to-digital quantizer. This connection was first established in \cite{jacques2013robust}. Uniform hyperplane tessellations are useful in the analysis of reconstruction algorithms for one-bit compressed sensing, in particular in proving uniform recovery results with robustness to corruptions occurring during quantization \cite{plan2012robust}. Third, uniform hyperplane tessellation results for Gaussian hyperplanes were recently used to derive bounds on the Lipschitz constants of deep ReLU neural networks with Gaussian weights \cite{dirksen2025near}.\par
Uniform hyperplane tessellation results for subsets of the sphere using standard Gaussian hyperplanes were first studied in the special case of the set of unit norm $s$-sparse vectors in \cite{jacques2013robust} and later for general $K$ in \cite{plan2014dimension}. The guarantees in \cite{plan2014dimension} were improved in \cite{oymak2015near}, but this still left a gap to the conjectured optimal result stated in \cite{oymak2015near,plan2014dimension}. It was already known that $m\geq C \delta^{-2} p(K)^2$ is sufficient for other complexity parameters $p(K)$ (that are possibly different from the Gaussian complexity) for subsets of the sphere that admit additional structure (see, e.g., \cite[Theorem 2.6]{oymak2015near}). For instance, for finite sets it is easy to show that $m\geq C\delta^{-2}\log(|K|)$, where $|K|$ is the cardinality of $K$, suffices. The latter condition was also established to be necessary using communication complexity bounds \cite{yi2015binary}.\par 
Using homogenous hyperplanes, which pass through the origin, one cannot separate points lying on a ray emanating from the origin. In particular, one can only hope to produce uniform tessellations of subsets of the sphere. As was already observed in \cite{plan2014dimension}, one can produce uniform tessellations of general Euclidean sets by adding random shifts to the hyperplanes. 
Closely connected to our work, \cite{dirksen2022sharp} studied uniform tessellations with hyperplanes with i.i.d.\ standard Gaussian directions and i.i.d.\ shifts that are uniformly distributed on $[-\lambda,\lambda]$ (with $\lambda$ sufficiently large) and derived a version of Theorem~\ref{thm:Oymak-Recht} for general bounded Euclidean sets. Moreover, they showed that their result was \emph{optimal}. As part of the optimality proof, they constructed a set $K\subset \R^n$ such that $m\sim \delta^{-3}\sqrt{\log(1/\delta)}w_*(K)^2$ hyperplanes are necessary and sufficient to induce a Euclidean $\delta$-uniform hyperplane tesselation \cite[Theorem 4.6]{dirksen2022sharp} if $\lambda\sim\sqrt{\log(1/\delta)}$. This result, however, did not provide an answer to \sjoerdnew{the conjectures of \cite{oymak2015near,plan2014dimension}}, as it considered hyperplanes with (sufficiently large) uniformly distributed shifts and the constructed $K$ is a convex set that, in particular, is not a subset of the sphere. In this work we revisit a `lifting argument' of  \cite{plan2014dimension} to connect spherical uniform hyperplane tessellations to Euclidean uniform hyperplane tessellations and adapt the arguments of \cite{dirksen2022sharp} to derive a necessary condition in the spherical setting. In particular, the set in Theorem~\ref{thm:main} is constructed using a suitable lifting. The sufficient condition in Theorem~\ref{thm:main} is derived directly from Theorem~\ref{thm:Oymak-Recht}, showing that the latter result cannot be improved further in general.       

\section{Lifting argument}

In this section we revisit and slightly sharpen a lifting argument introduced in \cite[Section 6]{plan2014dimension}. The goal is to derive a uniform tessellation result for a general set $K \subset \mathbb{R}^n$ with respect to the Euclidean distance by applying a spherical tessellation result to a \emph{lifting} of $K$ to the unit sphere in $\R^{n+1}$. 
Let us consider a general set $K \subset \mathbb{R}^n$ and denote $\mathrm{rad}(K) := \sup_{x \in K} \|x\|_2$. We can \emph{lift} this set by considering, for a given $\lambda>0$, the map 
$$Q_{\lambda} : \R^n \to \mathbb{S}^n, \ Q_{\lambda}(x) = \frac{x \oplus \lambda}{\|x \oplus \lambda\|_2}.$$
Consider $g_1,\ldots,g_m\in \R^n$ and $\tau_1,\ldots,\tau_m$, set $g_i' := (g_i, \tau_i)\in \R^{n+1}$, and let $G'=[G|\tau]\in \R^{m\times (n+1)}$ be the matrix with rows $g_i'$. Observe that for any $x\in K$,
\begin{align*}
\mathrm{sign}(G'Q_{\lambda}(x)) & = \left(\mathrm{sign}\left(\left\langle \frac{x \oplus \lambda}{\|x \oplus \lambda\|_2} , g_i' \right\rangle\right)\right)_{i = 1}^m \\
    &= \left(\mathrm{sign}(\langle(x \oplus \lambda), g_i'\rangle\right)_{i = 1}^m \\
    &= \left(\mathrm{sign}(\langle x, g_i\rangle  + \lambda \tau_i)\right)_{i = 1}^m.
\end{align*}
Hence, 
\begin{align*}
\hat{d}_{G',\lambda}(x,y) & := d_H(\mathrm{sign}(G'Q_{\lambda}(x)), \mathrm{sign}(G'Q_{\lambda}(y))) \\
& = \frac{1}{m} \sum_{i = 1}^m 1\{\mathrm{sign}\left(\langle x, g_i\rangle  + \lambda \tau_i \right) \neq \mathrm{sign}\left(\langle y, g_i\rangle  + \lambda \tau_i \right)\},
\end{align*}
is the fraction of \emph{affine} hyperplanes $H(g_i, \lambda \tau_i) := \{x : \langle x,g_i \rangle + \lambda \tau_i = 0\}$ that separate $x$ and $y$. Proposition~\ref{prop:inducedtessellation} will roughly show that if $G'$ induces a spherical $\delta^2$-tessellation of $Q_{\lambda}(K)$, then the affine hyperplanes will produce a $\delta$-uniform tessellation of $K$ with respect to the Euclidean metric (up to scaling). The proof relies on the following lemma, which is a slight modification of an argument in \cite{plan2014dimension}. The tighter bound stated here will be needed later to estimate the covering number of liftings of sets with small radius.   
\begin{lemma}
\label{lm:distancelift}
\sjoerd{Let $K\subset \R^n$ and suppose that $\mathrm{rad}(K)\leq \lambda$. Then, for any $x,y \in K$,} 
    $$\left|\|Q_{\lambda}(x) - Q_{\lambda}(y)\|_2 - \frac{1}{\lambda}\|x - y\|_2\right| \leq 4\mathrm{rad}(K)^{\sjoerd{2}}\lambda^{-2}.$$
\end{lemma}
\begin{proof}
    Put $r:= \mathrm{rad}(K)$. By the (reverse) triangle inequality
    \begin{align*}
        & \left|\|Q_{\lambda}(x) - Q_{\lambda}(y)\|_2 - \frac{1}{\lambda}\|x - y\|_2\right| \\
        &\qquad = \left|\|Q_{\lambda}(x) - Q_{\lambda}(y)\|_2 - \frac{1}{\lambda}\|x \oplus 0 - y\oplus 0\|_2\right| \\
        &\qquad \leq \left\|\frac{x}{\|x \oplus \lambda\|_2} - \frac{x}{\lambda}\right\|_2 + \left\|\frac{y}{\|y \oplus \lambda\|_2} - \frac{y}{\lambda}\right\|_2 + \left|\frac{\lambda}{\|x \oplus \lambda\|_2} - \frac{\lambda}{\|y \oplus \lambda\|_2}\right|\\
        &\qquad \leq \|x\|_2 \left| \frac{1}{\|x \oplus \lambda\|} - \frac{1}{\lambda}\right| + \|y\|_2 \left| \frac{1}{\|y \oplus \lambda\|} - \frac{1}{\lambda}\right| + \lambda \left|\frac{1}{\|x \oplus \lambda\|_2} - \frac{1}{\|y \oplus \lambda\|_2}\right|.
    \end{align*}
\sjoerd{Lemma~\ref{lem:lambdaEst} implies that for any $z\in K$}
    $$\left|\frac{1}{\|z \oplus \lambda\|_2} -\frac{1}{\lambda}\right| \leq r^2\lambda^{-3}.$$ 
    \sjoerd{Combining this with the triangle inequality, we find}
    $$\left|\|Q_{\lambda}(x) - Q_{\lambda}(y)\|_2 - \frac{1}{\lambda}\|x - y\|_2\right| \leq r^{\sjoerd{3}}\lambda^{-3} + r^{\sjoerd{3}}\lambda^{-3} + 2r^2\lambda^{-2} \leq 4r^{\sjoerd{2}} \lambda^{-2},$$
    as $\lambda \geq r$. This concludes the proof. 
\end{proof}
The following result was obtained in \cite[Equation (6.7)]{plan2014dimension}. We provide a proof for the sake of completeness. 
\begin{proposition}\label{prop:inducedtessellation}
There exist absolute constants $C,c > 0$ such that the following holds. Let $K$ be such that $\mathrm{rad}(K) \leq 1$ and $0 < \delta < c$. If $G'=[G|\tau]\in \R^{m\times(n+1)}$ induces a spherical $\frac{\delta^2}{4\pi C}$-uniform tessellation of $Q_{\lambda}(K)$ with $\lambda = 2C\delta^{-1}$, then the pair $(G,\lambda\tau)$ induces a Euclidean $\delta$-uniform tessellation of $K$ in the sense that
$$\left|\pi \lambda \hat{d}_{G',\lambda}(x,y) - \|x - y\|_2\right| \leq \delta \text{ for all } x,y \in K.$$  
\end{proposition}
\begin{proof}
\sjoerd{Let $\lambda\geq 1$}. Using Lemma \ref{fct:almostlip} (with $w = Q_{\lambda}(x), z = Q_{\lambda}(y)$) and Lemma \ref{lm:distancelift} yields
\begin{align}
& |\pi d_{\mathbb{S}^n}(Q_{\lambda}(x),Q_{\lambda}(y)) - \frac{1}{\lambda}\|x -y\|_2| \\
& \qquad \leq \left|\pi d_{\mathbb{S}^n}(Q_{\lambda}(x),Q_{\lambda}(y)) - \|Q_{\lambda}(x) -Q_{\lambda}(y)\|_2\right| + \left| \|Q_{\lambda}(x) - Q_{\lambda}(y)\| - \frac{1}{\lambda}\|x -y \|_2\right| \notag \\
&\qquad \leq C_0\|Q_{\lambda}(x) - Q_{\lambda}(y)\|_2^2 + \left| \|Q_{\lambda}(x) - Q_{\lambda}(y)\| - \frac{1}{\lambda}\|x -y \|_2\right| \notag\\
&\qquad \leq C_0 \left(\frac{1}{\lambda} \|x - y\|_2 + \left|\|Q_{\lambda}(x) - Q_{\lambda}(y)\|_2 - \frac{1}{\lambda}\|x - y\|_2\right|\right)^2 + 4\lambda^{-2} \notag\\
&\qquad \leq C_0\left(\frac{2}{\lambda} + 4\lambda^{-2}\right)^2 + 4\lambda^{-2} \leq C_1 \lambda^{-2}. \label{eq:distancesphereeuc}
\end{align}
\sjoerd{If $G'=[G|\tau]\in \R^{m\times(n+1)}$ induces a spherical $\delta_0$-uniform tessellation of $Q_{\lambda}(K)$, i.e.,}
$$\left|\pi \hat{d}_{G',\lambda}(x,y) - \pi d_{\mathbb{S}^n}(Q_{\lambda}(x), Q_{\lambda}(y))\right| \leq \pi \delta_0,$$
\sjoerd{then \eqref{eq:distancesphereeuc} yields}
\begin{align*}
& \left|\pi \hat{d}_{G',\lambda}(x,y) - \frac{1}{\lambda}\|x - y\|_2\right| \\
& \qquad \leq \left|\pi \hat{d}_{G',\lambda}(x,y) - \pi d_{\mathbb{S}^n}(Q_{\lambda}(x), Q_{\lambda}(y))\right| + \left|\pi d_{\mathbb{S}^n}(Q_{\lambda}(x),Q_{\lambda}(y)) - \frac{1}{\lambda}\|x -y\|_2\right| \leq \pi \delta_0 + C_1 \lambda^{-2}
\end{align*}
and hence
$$\left|\pi \lambda \hat{d}_{G',\lambda}(x,y) - \|x - y\|_2\right| \leq \pi \delta_0 \lambda + C_1 \lambda^{-1}.$$
The result follows by setting $\lambda := 2C_1\delta^{-1}$ and $\delta_0 := \frac{\delta^2}{4\pi C_1}$. 
\end{proof}
\sjoerd{Although Proposition~\ref{prop:inducedtessellation} states that a spherical tessellation of the lifted set yields a Euclidean tessellation of the original set, the distortion deteriorates from $\delta^2$ to $\delta$. At the same time, the lifted set $Q_{\lambda}(K)$ should become more one-dimensional as $\lambda$ increases and hence easier to tessellate using Gaussian hyperplanes. The following lemma formalizes this by showing that its Gaussian width decreases with $\lambda$. Below we use  
$$w(T) := \mathbb{E} \left(\sup_{x \in T} \langle x, g \rangle\right),$$
where $g$ is standard Gaussian, to denote the Gaussian width of $T\subset \R^n$.} 
\begin{lemma}
\label{lem:lifting1D}
Let $c>0$ be a constant and let $\lambda>0$. There exists a constant $C_1,C_2>0$ depending only on $c$ such that the following holds. If $K \subset \R^n$ satisfies \sjoerd{$\mathrm{rad}(K) \leq 1$ and $w_*(K) \geq c\lambda^{-1}$}, then 
    $$w(Q_{\lambda}(K)) \leq C_1\lambda^{-1}w_*(K).$$
\sjoerd{If additionally $w_*(K) \geq c\lambda$, then
$$w_*(Q_{\lambda}(K)) \leq C_2\lambda^{-1}w_*(K).$$
}
\end{lemma}
\vspace{-0.45cm}
\begin{proof}
Writing $g' = (g, \tau)$ with $g' \sim N(0,I_n)$ and $\tau \sim N(0,1)$ independent, we find
\begin{align}w(Q_{\lambda}(K)) &= \E\left(\sup_{x \in K} \left\langle \frac{x \oplus \lambda}{\|x \oplus \lambda\|_2}, g' \right\rangle\right) \notag \\
    &= \E\left(\sup_{x \in K} \left\langle \frac{x}{\|x \oplus \lambda\|_2}, g \right\rangle + \tau \frac{\lambda}{\|x \oplus \lambda\|_2}\right) \notag \\
    &\leq \E\left(\sup_{x \in K} \left|\left\langle \frac{x}{\|x \oplus \lambda\|_2}, g \right\rangle\right|\right) + 
    \mathbb{E}\left(\sup_{x \in K} \tau \frac{\lambda}{\|x \oplus \lambda\|_2}\right) \notag \\
    &\leq \frac{1}{\lambda} w_*(K) + \mathbb{E}\left(\sup_{x \in K} \tau 1_{\{\tau > 0\}} \frac{\lambda}{\|x \oplus \lambda\|_2}\right)+ \mathbb{E} \left(\sup_{x \in K} \tau 1_{\{\tau \leq 0\}} \frac{\lambda}{\|x \oplus \lambda\|_2}\right) \notag\\
        &\leq \frac{1}{\lambda} w_*(K) + \mathbb{E}(\tau 1_{\{\tau > 0\}}) + \mathbb{E} \left(\tau 1_{\{\tau \leq 0\}} \frac{\lambda}{\sqrt{1+\lambda^2}}\right) \label{eq:rad} \\
        &\leq \frac{1}{\lambda} w_*(K) + \mathbb{E}(\tau 1_{\{\tau > 0\}}) + \mathbb{E} \left(-\tau 1_{\{\tau \geq 0\}} \frac{\lambda}{\sqrt{1+\lambda^2}}\right) \label{eq:sym} \\
    &\leq \frac{1}{\lambda} w_*(K) + \mathbb{E}(\lambda |\tau|) \left(\frac{1}{\lambda} - \frac{1}{\sqrt{1 + \lambda^2}}\right)\notag\\
    &\leq \frac{1}{\lambda} w_*(K) + \frac{1}{\lambda^{2}} \E(|\tau|) \lesssim \frac{1}{\lambda} w_*(K), \label{eq:lambdadiff} 
\end{align}
where in \eqref{eq:rad} we used that $\mathrm{rad}(K) \leq 1$, \eqref{eq:sym} follows from symmetry of $\tau$, and in \eqref{eq:lambdadiff} we applied \sjoerd{Lemma~\ref{lem:lambdaEst}} and used $w_*(K) \gtrsim \frac{1}{\lambda}$. \sjoerd{The second assertion follows immediately from the fact that 
$$w_*(Q_{\lambda}(K))\sim w(Q_{\lambda}(K))+1,$$
see, e.g., \cite[Exercise 7.6.9]{vershynin2018high}.}  
\end{proof}
In \cite[Section 6]{plan2014dimension} it was already observed that $w_*(Q_{\lambda}(K)) \lesssim w_*(K)$ if $\lambda\geq 1$. Crucial for our argument is that by choosing $\lambda \sim \delta^{-1}$ in Lemma~\ref{lem:lifting1D} we obtain
$$w_*(Q_{\lambda}(K)) \lesssim w_*(K)\lambda^{-1} \sim \delta w_*(K)$$
if $w_*(K)\gtrsim \delta^{-1}$. As guarantees for Gaussian hyperplane tessellations scale in terms of the square of the Gaussian complexity, this will allows us to shave off a factor $\delta^2$ from our distortion dependence (see \eqref{eqn:deltaShaving} below). 

\subsection{A necessary condition for Euclidean tessellations}

Thanks to the lifting argument, we can now derive a necessary condition on the number of \emph{homogeneous} Gaussian hyperplanes needed to induce a \emph{spherical} uniform tessellation by establishing a necessary condition on the number of \emph{affine} Gaussian hyperplanes needed to induce a \emph{Euclidean} uniform tessellation. We will establish the latter by adapting an argument based on the Dvoretzky-Milman theorem from \cite[Theorem 1.10]{dirksen2022sharp}, who considered hyperplanes with Gaussian directions and uniformly distributed shifts. We will use the following upper bound on the inverse cdf of a folded standard normal random variable. 
\begin{lemma}\label{lm:inversecumgauss}
    Let $\sjoerd{\tau} \sim N(0,1)$ and let $\gamma$ be the number such that $P(|\sjoerd{\tau}| \leq \gamma) = \frac{2k}{m}$ with $k \leq \frac{m}{6}$. \sjoerd{Then,} 
	\begin{equation}
	\label{eqn:inversecumgauss}
	\sqrt{\frac{\pi}{2}} \frac{2k}{m} \leq \gamma \leq 2 \sqrt{\frac{\pi}{2}} \frac{2k}{m}.
	\end{equation}
\end{lemma}
\begin{proof}
The lower bound is clearly obtained by considering	
	\begin{align*}
		\mathbb{P}(|\sjoerd{\tau}| \leq \gamma) = \sqrt{\frac{2}{\pi}}\int_0^{\gamma} \exp(-t^2/2) dt \leq \sqrt{\frac{2}{\pi}} \gamma < \frac{2k}{m}.
	\end{align*}
	whenever $\gamma < \sqrt{\frac{\pi}{2}} \frac{2k}{m}$.\par
	Let us now derive the upper bound on $\gamma$. We first note that
	\begin{align*}
		\sqrt{\frac{2}{\pi}} \int_0^\gamma \exp(-t^2/2)dt \geq \sqrt{\frac{2}{\pi}} \int_0^\gamma 1 - \frac{t^2}{2} dt = \sqrt{\frac{2}{\pi}}\left[ \gamma - \frac{\gamma^3}{6}\right].
	\end{align*}	
	If we assume $\gamma \in [0,1]$, then $\gamma^3 \leq \gamma^2$ and hence\sjoerd{
	\begin{align}
	\label{eqn:cdfFoldedLB}
	\mathbb{P}(|\sjoerd{\tau}| \leq \gamma) & \geq \sqrt{\frac{2}{\pi}} \left[\gamma - \frac{\gamma^2}{6}\right] = \frac{2k}{m} - \frac{1}{6}\sqrt{\frac{2}{\pi}}\left[\gamma^2-6\gamma +6 \sqrt{\frac{\pi}{2}} \frac{2k}{m}\right] = \frac{2k}{m} - \frac{1}{6}\sqrt{\frac{2}{\pi}}(\gamma-\gamma_+)(\gamma-\gamma_-),
	\end{align}
	where 
	$$\gamma_{\pm}=3 \pm \sqrt{9 - 6 \sqrt{\frac{\pi}{2}} \frac{2k}{m}}.$$
	The right hand side of \eqref{eqn:cdfFoldedLB} is at least $2k/m$ if $\gamma_-\leq \gamma\leq \gamma_+$. Since $3 - \sqrt{9 - 6t} \leq 2t$ for all admissible $t\geq 0$, we see that 
	$$\gamma_-\leq 2 \sqrt{\frac{\pi}{2}} \frac{2k}{m}.$$
	In conclusion, $\mathbb{P}(|\sjoerd{\tau}| \leq \gamma) \geq 2k/m$ if 
$$2 \sqrt{\frac{\pi}{2}} \frac{2k}{m}\leq \gamma\leq 1.$$	 
Thus, the upper bound in \eqref{eqn:inversecumgauss} follows if} $4 \sqrt{\tfrac{\pi}{2}} \frac{k}{m} \leq 1$, which certainly holds when $k \leq \frac{m}{6}.$
\end{proof}
\begin{proposition}\label{prop:shiftbound}
Let $(\lambda_i)_{i = 1}^m$ be i.i.d. $N(0,\lambda^2)$ random variables. Then, \sjoerd{for any $k \leq \frac{m}{6}$ there exist} at least $k$ indices such that $|\lambda_i| \leq  4 \sqrt{\frac{\pi}{2}} \lambda\frac{k}{m}$ with probability at least $1 - 2\exp(-ck)$. 
\end{proposition}
\begin{proof}
    Let  $X = \sum_{i = 1}^m 1_{\{|\lambda_i| \leq 4 \sqrt{\frac{\pi}{2}} \lambda \frac{k}{m}\}}$ be the number of indices that satisfy the bound.
    This is a sum of i.i.d.\ Bernoulli random variables with mean $\mu$ where $\mu := \mathbb{P}(|\lambda_i| \leq 4 \sqrt{\frac{\pi}{2}} \lambda \frac{k}{m}) =  \mathbb{P}(|\frac{\lambda_i}{\lambda}| \leq 4 \sqrt{\frac{\pi}{2}} \frac{k}{m}) \geq \frac{2k}{m}$ by Lemma \ref{lm:inversecumgauss} as $\frac{\lambda_i}{\lambda}$ has a standard normal distribution. Therefore, by the Chernoff bound
    $$\mathbb{P}(X \geq k) \geq \mathbb{P}\left(X \geq \frac{1}{2} m \mu\right) \geq 1 - 2\exp(-c k)$$
    as $m \mu \geq m \frac{2k}{m} = 2k$.  
\end{proof}
Finally, we will use the Dvoretzky-Milman theorem. Let
$$
d_*(T) = \left(\frac{w_*(T)}{\mathrm{rad}(T)}\right)^2
$$
denote the \emph{Dvoretzky-Milman dimension} (or \emph{stable dimension}) of $T$. Recall that a set $K\subset \R^n$ is called a convex body if it is a convex, centrally-symmetric ($K=-K$) set with a nonempty interior.
\begin{theorem} \label{thm:DvorMil}
There are absolute constants $c_1,c_2,$ and $c_3$ such that the following holds. Let $K \subset \R^n$ be a convex body and let $k\leq c_1d_*(K)$. If $G\in \R^{k\times n}$ is standard Gaussian, then with probability at least $1-2\exp(-c_2 d_*(K))$,
$$
B_2^k(0;c_3 w_*(K)) \subset GK.
$$
\end{theorem}
A proof of Theorem~\ref{thm:DvorMil} can be found in, e.g., \cite{mendelson2016dvoretzky} and \cite[Section 11.3]{vershynin2018high}. We are now ready to establish the necessary condition for Euclidean uniform tessellations. 
\begin{theorem}
\label{thm:lowerbound}
There exist absolute constants $c_1,c_2,c_3,c_4,c_5,c_6>0$ such that the following holds. Assume that \sjoerd{$m \leq c_1\lambda \delta^{-3} w_*(K \cap B_2^n(0;\delta))^2$. Let $G\in \R^{m\times n}$ and $\tau\in \R^m$ be standard Gaussian and independent. Let $K$ be a convex body \sjoerd{with $3\delta\leq \mathrm{rad}(K) \leq 1$} and let \sjoerd{$\lambda \geq c_2\delta$}. Set
$$k^*:=\min\left\{c_3w_*(K \cap B_2^n(0;\delta))^2\delta^{-2}, c_4\left(\frac{m}{\lambda}w_*(K \cap B_2^n(0;\delta))\right)^{2/3},c_5m\right\}.$$
Then, with probability at least $1-4e^{-c_6k^*}$, the pair $(G,\lambda \tau)$ fails to induce a Euclidean $\delta$-tessellation of $K$, i.e., 
$$\left|\pi \lambda \hat{d}_{G',\lambda}(x,y) - \|x - y\|_2\right| > \delta$$
for certain $x,y \in K$. 
} 
\end{theorem}
\begin{proof}
\sjoerd{Set $\lambda_i=\lambda\tau_i$. It suffices to find an $\sjoerd{x^*}\in K \cap B_2^n(0;\delta)$ such that
\begin{equation}
\label{eqn:lowerboundToProve}
\left|\pi \lambda \frac{1}{m} \sum_{i = 1}^m 1\{\mathrm{sign}\left(\langle g_i, \sjoerd{x^*}\rangle + \lambda_i\right) \neq \mathrm{sign}\left(\langle g_i, 0 \rangle + \lambda_i\right)\} - \|\sjoerd{x^*} - 0\|_2\right| >  2\delta.
\end{equation}
If $k \leq \frac{m}{6}$, then by Proposition~\ref{prop:shiftbound} there exists a set $I'$ of at least $k$ indices such that $|\lambda_i| \leq 4 \sqrt{\frac{\pi}{2}} \lambda \frac{k}{m}$ for all $i \in I'$ with probability at least $1 - \exp(-c_0k)$ where $c_0$ is an absolute constant and $k$ will be specified later. From now on, condition on this event, which only depends on the shifts $\lambda_i$ and not on the $g_i$. Take the first $k$ indices when $I'$ is sorted in increasing order and call this index set $I$. Note that $\lambda_I := (\lambda_i)_{i \in I}$ satisfies     
\begin{equation}
\label{eqn:lambdaIbound}
\|\lambda_I\|_2 \lesssim \frac{k^{3/2}\lambda}{m}.
\end{equation}
Let $G_I\in \R^{k\times m}$ be the Gaussian random matrix with rows $g_i$, $i\in I$. Since $3\delta\leq \mathrm{rad}(K)$,  
$$d^*(K \cap B_2^n(0;\delta))\sim \frac{w_*(K \cap B_2^n(0;\delta))^2}{\delta^2}.$$ 
Hence,} by the Dvoretzky-Milman theorem (Theorem~\ref{thm:DvorMil}), there are absolute constants $c_1,c_2,c_3>0$ such that if 
$$k \leq c_1\frac{w_*(K \cap B_2^n(0;\delta))^2}{\delta^2},$$   
then
\begin{equation}
\label{eqn:DMinc}
B_2^{\sjoerd{k}}(0; c_2w_*(K \cap B_2^n(0;\delta))) \subset \sjoerd{G_I}(K \cap B_2^n(0;\delta))
\end{equation}
with probability at least \sjoerd{$1 - 2 \exp(-c_3w_*(K \cap B_2^n(0;\delta))^2/\delta^2)$}. Moreover, if 
\begin{align}
\frac{k^{3/2}\lambda}{m} \leq c_4 w_*(K \cap B_2^n(0;\delta)),\label{eq:changeshifthyp}
\end{align}
for a small enough absolute constant $c_4>0$, then we can find $\sjoerd{x^*} \in K\cap B_2^n(0;\delta)$ such that 
$$\mathrm{sign}(\langle \sjoerd{x^*}, g_{\sjoerd{i}}\rangle + \lambda_i) \neq \mathrm{sign}(\lambda_i), \qquad \text{\sjoerd{for all $i\in I$}}.$$
To see this, note first that for any $\varepsilon>0$ 
$$\mathrm{sign}((-\lambda_i - \varepsilon\mathrm{sign}(\lambda_i)) + \lambda_i) \neq \mathrm{sign}(\lambda_i), \qquad \text{\sjoerd{for all $i\in I$}}.$$
Moreover, by \eqref{eqn:lambdaIbound}, the absolute constant $c_4$ in \eqref{eq:changeshifthyp} can be chosen small enough so that 
$$\|\lambda_I\|_2 < c_2w_*(K \cap B_2^n(0;\delta))$$
with $c_2$ as in \eqref{eqn:DMinc}. We can then find an $\varepsilon > 0$ sufficiently small such that 
$$\|(-\lambda_i - \varepsilon\mathrm{sign}(\lambda_i))_{i \in I}\|_2 \leq c_2 w_*(K \cap B_2^n(0;\delta)).$$
By \eqref{eqn:DMinc}, we can then represent $(-\lambda_i - \varepsilon\mathrm{sign}(\lambda_i))_{i \in I}$ as $(\langle \sjoerd{x^*}, g_i\rangle)_{i \in I}$ for some $\sjoerd{x^*} \in K\cap B_2^n(0;\delta)$.\par
In conclusion, there are absolute constants $c_5,c_6,c_7>0$ such that if
$$k^* = \min\left\{c_5w_*(K \cap B_2^n(0;\delta))^2\delta^{-2},c_6\left(\frac{m}{\lambda}w_*(K \cap B_2^n(0;\delta))\right)^{2/3},\frac{m}{6}\right\}$$
then with probability at least 
$$(1 - 2 \exp(-c_3w_*(K \cap B_2^n(0;\delta))^2/\delta^2))(1-2\exp(-c_0k^*)) \geq 1-4\exp(-c_7k^*),$$
there exists an $\sjoerd{x^*} \in K\cap B_2^n(0;\delta)$ such that 
\begin{align*} 
& \pi \lambda \frac{1}{m} \sum_{i = 1}^m 1\{\mathrm{sign}\left(\langle g_i, \sjoerd{x^*}\rangle + \lambda_i\right) \neq \mathrm{sign}\left(\langle g_i, 0 \rangle + \lambda_i\right)\}\\
& \qquad \geq \pi \lambda \frac{k^*}{m} = \min\left\{c_4\pi\frac{\lambda w_*(K \cap B_2^n(0;\delta))^2}{m \delta^2},c_5\pi\frac{\lambda^{1/3}}{m^{1/3}} w_*(K \cap B_2^n(0;\delta))^{2/3},\frac{\lambda \pi}{6}\right\}.
\end{align*}
Hence, \eqref{eqn:lowerboundToProve} holds if the right hand side is at least $3\delta$, which is satisfied if $m \leq c_8\lambda w_*(K \cap B_2^n(0;\delta))^2\delta^{-3}$ with $c_8$ a small enough absolute constant and $\lambda\geq 6\delta/\pi$.
\end{proof}
\begin{remark}
\label{rem:mLowerBound}
\sjoerdnew{The probability estimate in Theorem~\ref{thm:lowerbound} is only nontrivial if $m$ is sufficiently large (and $w_*(K \cap B_2^n(0;\delta))$ is sufficiently large). Let us note that we may always assume that $m > \lambda/2$ if $\delta\leq 1$. Indeed, suppose that $\lambda/m \geq 2$ and consider an $x^*\in K$ with $3 \delta\leq \|x^*\|\leq 1$. Then for any realisation of $G$ and $\tau$ we have either $$\sum_{i = 1}^m 1\{\mathrm{sign}\left(\langle g_i, x^*\rangle + \lambda_i\right) \neq \mathrm{sign}\left(\langle g_i, 0 \rangle + \lambda_i\right)\} = 0$$ \emph{or} $$\sum_{i = 1}^m 1\{\mathrm{sign}\left(\langle g_i, x^*\rangle + \lambda_i\right) \neq \mathrm{sign}\left(\langle g_i, 0 \rangle + \lambda_i\right)\} \geq 1.$$
In the first case, \eqref{eqn:lowerboundToProve} holds as $\|x^*\|\geq 3\delta$. In the second case, \eqref{eqn:lowerboundToProve} holds for $\delta\leq 1$ as $\lambda/m \geq 2$ and $\|x^*\|\leq 1$. Hence, in both cases $(G,\lambda \tau)$ fails to induce a Euclidean $\delta$-tessellation of $K$.} 
\end{remark}

\subsection{Proof of Theorem~\ref{thm:main}}

We consider the set 
$$K(\delta) := [-3\delta, 3\delta] \oplus B_2^{n(\delta)}(0;\varepsilon(\delta)),$$
let $\lambda=\lambda(\delta)=2C_0\delta^{-1}$\color{black} and define 
$$S(\delta) = Q_{\lambda}(K(\delta)).$$
Observe that $K(\delta)$ is a convex body and $3 \delta \leq \mathrm{rad}(K(\delta)) \leq 1$. Recall that there exist absolute constants $c_1,C_1>0$ such that
$c_1 \sqrt{n} \leq w_*(B_2^n(0;1)) \leq C_1 \sqrt{n}$.
Hence,
\begin{align*}
    c_1 \varepsilon(\delta) \sqrt{n} \leq w_*(B_2^n(0;\varepsilon(\delta))) \leq C_1 \varepsilon(\delta) \sqrt{n}.
\end{align*}
Therefore, $w_*(B_2^n(0;\varepsilon(\delta))) \sim c(\delta)$ as long as $n  = n(\delta) \sim \varepsilon(\delta)^{-2} c(\delta)^2$. In particular, if $c(\delta) \sim \lambda$ and $\varepsilon(\delta) < \delta$ then 
\begin{equation}
\label{eqn:condGWGC}
\sjoerdnew{c_3}w_*(K(\delta) \cap B_2^n(0;\delta)) \geq w_*(K(\delta))\geq w_*(K(\delta)\cap B_2^n(0;\delta))\geq c_2 \lambda.
\end{equation}\color{black}
Set $\delta_0 :=\delta^2/(4\pi C_0)$ and suppose that 
$$m\leq c_4\delta_0^{-3} w_*(S(\delta))^2$$
and let $G\in \R^{m\times n}$ and $\tau\in \R^m$ be standard Gaussian and independent. By Proposition~\ref{prop:inducedtessellation}, if $G'=[G|\tau]$ would induce a spherical $\delta_0$-uniform tessellation of $S(\delta)$, then the pair $(G,\lambda \tau)$ would induce a Euclidean $\delta$-uniform tessellation of $K(\delta)$. However, by \eqref{eqn:condGWGC} and the second assertion in Lemma~\ref{lem:lifting1D} 
\begin{align}
\label{eqn:deltaShaving}
m & \leq c_4\delta_0^{-3} w_*(S(\delta))^2\leq c_4C_2\delta^{-6} \left(\frac{w_*(K(\delta))}{\lambda}\right)^2 \nonumber\\
& = c_4C_3\lambda\delta^{-3}w_*(K(\delta))^2 \leq c_4C_3\sjoerdnew{c_3^2}\sjoerdnew{\lambda\delta^{-3}} w_*(K(\delta) \cap B_2^n(0;\delta))^2. 
\end{align}
Hence, if $c_4$ is small enough (and $\delta<c_5$ so that $\lambda\geq c_6\delta$ for $c_6$ large enough), then this would contradict the event of Theorem~\ref{thm:lowerbound}, which happens with probability at least $1-4e^{-c_7k^*}$, where
$$k^*=\min\left\{c_8w_*(K(\delta) \cap B_2^n(0;\delta))^2\delta^{-2},c_9\left(\frac{m}{\lambda}w_*(K(\delta) \cap B_2^n(0;\delta))\right)^{2/3},c_{10}m\right\}.$$
This probability exceeds $0.99$ if $m\geq C_4$ for a sufficiently large $C_4$ and $\delta<c_{11}$ for a sufficiently small $c_{11}>0$. \sjoerdnew{By Remark~\ref{rem:mLowerBound}, we may assume that $m \geq \frac{\lambda}{2} \geq C_4$ if $c_{11}$ is sufficiently small}.\par
On the other hand, if we again set $\delta_0 :=\delta^2/(4\pi C_0)$, let $\varepsilon(\delta)=c_{12}\frac{\delta_0}{\log(\delta_0)}$\color{black} and assume 
$$m \geq C_5\delta^{-4} \log N(S(\delta);\varepsilon(\delta)) + C_5\delta^{-6}w((S(\delta) - S(\delta)) \cap B(0;\varepsilon(\delta)))^2,$$
then Theorem~\ref{thm:Oymak-Recht} implies that $[G|\tau]$ induces a spherical $\frac{\delta^2}{4\pi C_0}$-uniform tessellation of $S(\delta)$ with probability at least $0.99$\color{black}. It remains to show that this condition holds if $m\geq C_6\delta^{-6} w_*(S(\delta))^2$. Clearly, it suffices to show that 
$$\log N(S(\delta);\varepsilon(\delta))\leq \delta^{-2}.$$
To see this, consider a regular grid of $[-3\delta, 3\delta]$ of mesh size $\varepsilon(\delta)$. Clearly, there are at most $1/\varepsilon(\delta)$ grid points if $\delta\leq 1/6$. For any $x=(x_1,\hat{x})\in K(\delta)$, where $x_1\in [-3\delta,3\delta]$ and $\hat{x}\in B_2^{n(\delta)}(0;\varepsilon(\delta))$. Let $\pi_1(x)$ be the grid point closest to $x_1$ and let $\pi(x)=(\pi_1(x),0)\in \R^{n(\delta)+1}$. Using Lemma~\ref{lm:distancelift} and that $\mathrm{rad}(K) \leq 10 \delta$ for $\delta<c_{14}\color{black}$ with $c_{14}\color{black}$ sufficiently small, we find
\begin{align*}
\|Q_{\lambda}(x) - Q_{\lambda}(\pi(x))\| & \leq 4\lambda^{-2} \mathrm{rad}(K(\delta))^{\sjoerd{2}}+ \lambda^{-1} \|x - \pi(x)\|_2\\
& \leq 4\lambda^{-2} \mathrm{rad}(K(\delta))^{\sjoerd{2}}+ \lambda^{-1} \|x_1 - \pi_1(x)\|_2 + \lambda^{-1}\|\hat{x}\|_2\\
& \leq 400\lambda^{-2} \delta^{\sjoerd{2}} + \frac{\delta}{2C_0} \varepsilon(\delta) +  \frac{\delta}{2C_0} \varepsilon(\delta) \leq \varepsilon(\delta)
\end{align*}
for $\delta<c_{14}\color{black}$ with $c_{14}\color{black}$ sufficiently small. Thus, 
$$\log N(S(\delta);\varepsilon(\delta))\leq \log(1/\varepsilon(\delta))\leq \delta^{-2}$$
if $c_{14}\color{black}$ is small enough.\par 
To complete the proof, we apply the results obtained above with $\delta$ replaced by $\sqrt{\delta\cdot  4\pi C_0}$. 

\appendix

\section{Two technical observations}

The following inequality states that the difference between the unnormalised spherical distance and the Euclidean distance on the sphere decays at the rate of the square of the Euclidean distance. We provide a proof for the convenience of the reader. 
\begin{lemma}\label{fct:almostlip}
There is a constant $C_0 > 0$ such that for all $w,z \in \mathbb{S}^{n-1}$, 
$$ \left|\pi d_{\mathbb{S}^{n-1}}(w,z) - \|w - z\|_2\right| \leq C_0 \|w - z\|_2^2.$$
\end{lemma}
\begin{proof}
    We note that $\|w - z\|_2^2 = 2 - 2 \langle w,z\rangle$ and $\pi d_{\mathbb{S}^{n-1}}(w,z) = \arccos(\langle w,z \rangle)$. Thus, this statement reduces to show that there exists an absolute constant $C_0$ such that for all $x \in [-1,1]$,
    \begin{align}
    \label{eqn:arccosBound}
    \left|\arccos(x) - \sqrt{2 - 2x}\right| \leq C_0(2 - 2x).
    \end{align}
    To see this, we note that the function inside the absolute value is differentiable on $]-1,1[$ with derivative equal to $\frac{-1}{\sqrt{1 - x^2}} + \frac{1}{\sqrt{2 - 2x}}$. In this range, we can simplify this expression to 

    $$\frac{1}{\sqrt{1-x}}\left(-\frac{1}{\sqrt{1+x}} + \frac{1}{\sqrt{2}}\right).$$
    The claim is that $\lim_{x \uparrow 1} \frac{1}{\sqrt{1-x}}\left(-\frac{1}{\sqrt{1+x}} + \frac{1}{\sqrt{2}}\right) = 0$. To see this, one can look at the square of this expression $\frac{\left(-\frac{1}{\sqrt{1+x}} + \frac{1}{\sqrt{2}}\right)^2}{1-x}$ and then use L'Hôpital's rule to see that the limit is equal to 0 as $x \to 1$. It follows that the derivative of $\arccos(x) - \sqrt{2 - 2x}$ can be extended to a continuous function on $[0,1]$. For $x \in [-1,0]$, \eqref{eqn:arccosBound} is obvious by possibly increasing the absolute constant. On $[0,1]$ the absolute value of the derivative is continuous and therefore attains a maximal value $M$. It follows by the mean value theorem that 
$$|\arccos(x) - \sqrt{2 - 2x}| \leq M (1 - x).$$
Thereby, the statement is proven. 
\end{proof}
\sjoerd{The following inequality is a straightforward consequence of the mean value theorem.
\begin{lemma}
\label{lem:lambdaEst}
For any $r>0$ and $\lambda>0$,
$$\left|\frac{1}{\sqrt{r^2 + \lambda^2}} - \frac{1}{\lambda}\right| \leq r^2 \lambda^{-3}.$$
\end{lemma}
\begin{proof}
Observe that 
$$\left|\frac{1}{\sqrt{r^2 + \lambda^2}} - \frac{1}{\lambda}\right| \leq \frac{1}{r} \left|\frac{1}{ \sqrt{1 + (\frac{\lambda}{r})^2}} - \frac{1}{\frac{\lambda}{r}}\right|.$$
Hence, it suffices to show for all $\mu>0$ that 
\begin{align*}
\left|\frac{1}{\sqrt{1 + \mu^2}} - \frac{1}{\mu}\right| \leq \mu^{-3}
\qquad \text{ or, equivalently, }\qquad 
\left|\frac{1}{\sqrt{1 + \mu^{-2}}} - 1\right| \leq \mu^{-2}.
\end{align*} 
By applying the mean value theorem to $f(x) := \frac{1}{\sqrt{1 + x}}$ shows that for any $x > 0$ there is some $\theta \in [0,x]$ so that
    $\left|\frac{1}{\sqrt{1 + x}} - 1\right| \leq |f'(\theta)|x\leq \frac{1}{2}x.$
Setting $x = \mu^{-2}$ yields the desired inequality. 
\end{proof}
}


\begin{thebibliography}{10}

\bibitem{boufounos20081}
P.~T. Boufounos and R.~G. Baraniuk.
\newblock 1-bit compressive sensing.
\newblock In {\em 2008 42nd Annual Conference on Information Sciences and
  Systems}, pages 16--21. IEEE, 2008.

\bibitem{dirksen2025near}
S.~Dirksen, P.~Finke, P.~Geuchen, D.~St{\"o}ger, and F.~Voigtlaender.
\newblock Near-optimal estimates for the {$\ell^{p}$}-{L}ipschitz constants of
  deep random {R}e{LU} neural networks.
\newblock {\em arXiv:2506.19695}, 2025.

\bibitem{dirksen2022sharp}
S.~Dirksen, S.~Mendelson, and A.~Stollenwerk.
\newblock Sharp estimates on random hyperplane tessellations.
\newblock {\em SIAM Journal on Mathematics of Data Science}, 4(4):1396--1419,
  2022.

\bibitem{dirksen2018fast}
S.~Dirksen and A.~Stollenwerk.
\newblock Fast binary embeddings with gaussian circulant matrices: improved
  bounds.
\newblock {\em Discrete \& Computational Geometry}, 60(3):599--626, 2018.

\bibitem{jacques2013robust}
L.~Jacques, J.~N. Laska, P.~T. Boufounos, and R.~G. Baraniuk.
\newblock Robust 1-bit compressive sensing via binary stable embeddings of
  sparse vectors.
\newblock {\em IEEE transactions on information theory}, 59(4):2082--2102,
  2013.

\bibitem{mendelson2016dvoretzky}
S.~Mendelson.
\newblock Dvoretzky type theorems for subgaussian coordinate projections.
\newblock {\em Journal of Theoretical Probability}, 29:1644--1660, 2016.

\bibitem{oymak2015near}
S.~Oymak and B.~Recht.
\newblock Near-optimal bounds for binary embeddings of arbitrary sets.
\newblock {\em arXiv:1512.04433}, 2015.

\bibitem{oymak2017near}
S.~Oymak, C.~Thrampoulidis, and B.~Hassibi.
\newblock Near-optimal sample complexity bounds for circulant binary embedding.
\newblock In {\em 2017 IEEE International Conference on Acoustics, Speech and
  Signal Processing (ICASSP)}, pages 6359--6363. IEEE, 2017.

\bibitem{plan2013one}
Y.~Plan and R.~Vershynin.
\newblock One-bit compressed sensing by linear programming.
\newblock {\em Communications on pure and Applied Mathematics},
  66(8):1275--1297, 2013.

\bibitem{plan2012robust}
Y.~Plan and R.~Vershynin.
\newblock Robust 1-bit compressed sensing and sparse logistic regression: A
  convex programming approach.
\newblock {\em IEEE Transactions on Information Theory}, 59(1):482--494, 2013.

\bibitem{plan2014dimension}
Y.~Plan and R.~Vershynin.
\newblock Dimension reduction by random hyperplane tessellations.
\newblock {\em Discrete \& Computational Geometry}, 51(2):438--461, 2014.

\bibitem{vershynin2018high}
R.~Vershynin.
\newblock {\em High-dimensional probability: An introduction with applications
  in data science}, volume~47.
\newblock Cambridge university press, 2018.

\bibitem{yi2015binary}
X.~Yi, C.~Caramanis, and E.~Price.
\newblock Binary embedding: Fundamental limits and fast algorithm.
\newblock In {\em International Conference on Machine Learning}, pages
  2162--2170. PMLR, 2015.

\bibitem{yu2014circulant}
F.~Yu, S.~Kumar, Y.~Gong, and S.-F. Chang.
\newblock Circulant binary embedding.
\newblock In {\em International conference on machine learning}, pages
  946--954. PMLR, 2014.

\end{thebibliography}
\end{document}